\documentclass[12pt]{article}
\usepackage[cp1251]{inputenc}
\usepackage{amssymb,latexsym,amsmath}

\usepackage{amsfonts}
\usepackage[english]{babel}
\usepackage{indentfirst}
\usepackage{graphicx,graphics,verbatim}
	\usepackage{cite}
	\usepackage{amsthm}

	\newtheorem{teo}{Theorem}

	\newtheorem{example}{Example}
	
	\newtheorem{zadacha}{Problem}
	
\begin{document}
		

\begin{center}		
	{\bf Hornbook of solving problems for mixed parabolic--hyperbolic equations} \\
\end{center}	

\begin{center}
		\textbf{Elina L. Shishkina}\\
	Voronezh State University, Voronezh, Russia\\
	Belgorod State University, Belgorod, Russia\\
shishkina@amm.vsu.ru\\
	\textbf{Azamat V. Dzarakhohov}\\
Gorsky State Agrarian University, Vladikavkaz, Russia\\
azambat79@mail.ru\\
\end{center}

	{\bf  The first author would like to thank the Isaac Newton Institute for Mathematical Sciences, Cambridge, for support and hospitality during the programme "Fractional differential equations"\, where work on this paper was undertaken. This work was supported by EPSRC grant no EP/R014604/1.}

\vskip 1cm


{\bf Abstract}. In this small paper, we study a boundary value problem for an equation of parabolic-hyperbolic type. The goal is to show how we can prove existence and uniqueness theorem for a regular
solution.  


\section{Introduction}

One of the most fascinating recent areas of partial differential equations is known as the theory of boundary value problems for equations of mixed type. 
Study of boundary-value problems for mixed-type equations are needed knowledge from the different fields of mathematics.  The proof of the uniqueness and existence of a solution is usually based on  methods of fractional differentiation, special functions, and integral equations. 
However, it was noticed that all parers, in which mixed equations of parabolic-hyperbolic type are studied, contain the same sequence of actions. Only the formulas in specific steps become more complicated. In this paper we would like to present this algorithm and illustrate it by the simplest example.


\section{Basic principles of solving problems for mixed parabolic--hyperbolic equations}

In this section, we consider the main approaches to solving problems for mixed parabolic--hyperbolic equations. 

Let us consider 
mixed parabolic--hyperbolic equation
\begin{equation}\label{SZ01}
	0=\begin{cases}
		L_1u(x,y),&y>0,\\
		L_2 u(x,y),&y<0,
	\end{cases}
\end{equation}
where $L_1$ is differential operator of parabolic type acting in variables $(x,y)$ and $L_2$ is differential operator of hyperbolic type also acting in variables   $(x,y)$. Equation \eqref{SZ01} is considered in a bounded domain $D$. Domain $D$ is divided by the $Ox$ axis into two non-empty domains  $D_1$ and $D_2$. The first equation $L_1u(x,y)=0$ is considered in the domain $D_1$. The second equation  $L_2u(x,y)=0$ is considered in the domain $D_2$. 
Some conditions are added to \eqref{SZ01} on all or part of the boundary of $D$. 
Usually, a solution should be continuous in $\overline{D}$ with some additional smoothness requirements in  $D_1$, $D_2$ and parts of their boundaries. Such a solution is called regular.
 As a rule, in such problems solutions of the equations $L_1u(x,y)=0$ and $L_2u(x,y)=0$ are known separately, and the main problem is to glue these solutions and to obtain  continuous on $D$ function $u(x,y)$ which is satisfied all other necessary conditions.

Let $L_1$ and $L_2$ be non-degenerate differential operators, $L_1$ be a differential operator of parabolic type, and $L_2$ be of second-order hyperbolic type. 
Different conditions are added to equation \eqref{SZ01} in the domains $D_1$ and $D_2$, which must be matched on the line $y=0$.
Smoothness requirements are imposed on these conditions, taking into account the continuous gluing of the solutions in $D_1$ and in $D_2$ along the line $y=0$.
Uniqueness and existence of the solution.
 to such problem usually can be proved by the following algorithm. 

\begin{center}
	{\bf Algorithm of proving the uniqueness and existence of a solution to the  mixed parabolic--hyperbolic equation.} 
\end{center}
\begin{description}
	\item[Step 1.] Assume that a solution $u(x,y)$ exists. We write conditions on the line $y=0$ in the form $u(x,0)=\tau(x)$, $u_y(x,0)= \nu(x)$ and compose for the functions $\tau(x)$, $\nu(x)$ an equation using the equality $L_1u(x,y)=0$ for $y=0$. This equation gives the first relation between the functions $\tau(x)$, $\nu(x)$. 
	\item[Step 2.] Solving for $y<0$ the problem  $L_2u(x,y)=0$, 	$u(x,0)=\tau(x)$, $u_y(x,0)=\nu(x)$ and using additional conditions we obtain the second  relation between the functions $\tau(x)$, $\nu(x)$.
	\item[Step 3.] Since we have two relations between   functions $\tau(x)$ and $\nu(x)$, it is usually possible to get the equation only for one function $\tau (x)$ or $\nu(x)$ from this system. Considering the problem for one of this function with homogeneous condition
	we obtain that it has only a trivial solution, identically equal to zero. Therefore, the second function under homogeneous conditions is identically equal to zero. That means that the solution $u(x,y)$ in $D_2$ is identically equal to zero under homogeneous conditions, and the uniqueness of $u(x,y)$ in $D_2$ is proved.
		\item[Step 4.] Studing the solution of the problem in a parabolic domain, we find that under homogeneous conditions this problem has only a trivial solution in $D_1$. So we prove the uniqueness of the solution in $D$.
			\item[Step 5.] To prove the existence of a solution, we again consider the problem for $\tau(x)$ or $\nu(x)$, but now with inhomogeneous conditions. Usually, for this solution the existence theorem is known. Also, it can often be written explicitly.
			If we can write explicit expressions for  $\tau(x)$ and $\nu(x)$ we can write explicit expression for $u(x,y)$ in each domain $D_1$ and $D_2$ and this solution $u(x,y)$ will be continuous in $D$.
\end{description}


\section{The simplest mixed parabolic-hyperbolic equation}
 
 In order to illustrate the algorithm from the previous section, let us consider the simplest mixed parabolic-hyperbolic equation
\begin{equation}\label{PG01}
	0=\begin{cases}
		u_{xx}-u_y,&y>0,\\
		u_{yy}-u_{xx},&y<0.
	\end{cases}
\end{equation}
in a simply connected domain $D$ of the plane of variables $x,y$ bounded by segments $AA_0$ of the line $x=0$, $A_0B_0$ of the line $x=1$, $y=1$ and  for $y<0$ by real characteristics $AC:x+y=0$, $BC:x-y=1$ of the equation \eqref{PG01} coming from the points $A(0,0)$, $B(1,0 )$. Let $D=D_1\cup D_2$, where $D_1$ is the parabolic part of $D$, $y>0$, and $D_2$ is the hyperbolic part of $D$, $y<0$.

A regular solution of \eqref{PG01} in $D$ is a function $u(x,y){\in}C(\overline{D}){\cap} C^2(D)$ satisfying the equation \eqref{PG01}.

\begin{zadacha}\label{PG02}
	Find a solution $u(x,y)$ of the equation \eqref{PG01} that is regular in $D$ and satisfies the conditions
	\begin{equation}\label{PG03}
		u(0,y)=\varphi_0(y),\qquad 	u(1,y)=\varphi_1(y),\qquad y>0,
	\end{equation}
	\begin{equation}\label{PG04}
		au\left(\frac{x}{2},-\frac{x}{2}\right) +bu\left(\frac{x+1}{2},\frac{x-1}{2}\right)=\psi(x),\qquad 	0\leq x\leq 1,
	\end{equation}
where $\varphi_0(y)$, $\varphi_1(y)$, $\psi(x)$ are given functions, $a$, $b$ are given real constants, $a^2+b^2>0$.
\end{zadacha}

In order to match the conditions \eqref{PG03} and \eqref{PG04}, we multiply \eqref{PG04} first by $a$ and take a limit $x\rightarrow +0$, and then we multiply \eqref{PG04} by $b$ and take a limit $x\rightarrow 1-0$. 
Next we subtract the second equality from the first equality:
$$
-\left\{ \begin{array}{ll}
	a^2\lim\limits_{y\rightarrow +0}\varphi_0(y)  +abu\left(\frac{1}{2},-\frac{1}{2}\right)=a\psi(0);\\
	abu\left(\frac{1}{2},-\frac{1}{2}\right) +b^2\lim\limits_{y\rightarrow +0} \varphi_1(y)=b\psi(1).\end{array} \right.
$$
We obtain
\begin{equation}\label{PG05}
	a^2\lim\limits_{y\rightarrow +0} \varphi_0(y)-b^2 \lim\limits_{y\rightarrow +0} \varphi_1(y)=a\psi(0)-b\psi(1).
\end{equation}

\begin{teo}
If	$\varphi_0(y),\varphi_1(y)\in C[0,1]$, $\psi(x)\in C^1[0,1]\cap C^2(0,1)$, 
	and the conditions  $a\neq b$ 
	and \eqref{PG05}  are valid, then problem \ref{PG02} has a unique regular solution.
\end{teo}
\begin{proof} 
	\begin{description} 
	\item[Step 1.]	
	Suppose that there is a solution $u(x,y)$ to the problem \ref{PG02}.
	Let us prove that it is unique.
Let us introduce the notation
	\begin{equation}\label{PG06}
		\tau(x)=\lim\limits_{y\rightarrow +0} u(x,y),\qquad x\in[0,1],
	\end{equation}
	\begin{equation}\label{PG07}
		\nu(x)=\lim\limits_{y\rightarrow +0}u_y(x,y),\qquad x\in(0,1).
	\end{equation}
From the conditions of the problem we obtain
	\begin{equation}\label{PG08}
		\tau(0)=\lim\limits_{y\rightarrow +0}u(0,y)=\varphi_0(0),\qquad \tau(1)=\lim\limits_{y\rightarrow +0} u(1,y)=\varphi_1(0).
	\end{equation}
The relation between $\tau(x)$ and $\nu(x)$ brought from the parabolic region $D_1$ has the form
	\begin{equation}\label{PG09}
		\nu(x)=\tau''(x).
	\end{equation}
So we obtain the first relation between the functions $\tau(x)$, $\nu(x)$ in the form \eqref{PG09}. 

	\item[Step 2.] Solution to the Cauchy problem
	$$
	u_{xx}-u_{yy}=0,\qquad y<0,
	$$
$$
	u(x,0)=\tau(x),\qquad 	u_y(x,0)=\nu(x)
$$
is
	\begin{equation}\label{PG10}
		u(x,y)=\frac{\tau(x+y)+\tau(x-y)}{2}-\frac{1}{2}\int\limits_{x+y}^{x-y}\nu(\xi)d\xi.
	\end{equation}
From the condition \eqref{PG04} we obtain
	$$
	\frac{a}{2}\left(\tau(0)+\tau(x)-\int\limits_{0}^{x}\nu(\xi)d\xi \right)+\frac{b}{2}\left(\tau(1)+\tau(x)-\int\limits_{x}^{1}\nu(\xi)d\xi \right) =\psi(x)
	$$
Differentiating this equality with respect to $x$, we obtain
	\begin{equation}\label{PG11}
		(a+b)\tau'(x)-(a-b)\nu(x)=2\psi_1'(x).
	\end{equation}
Expressing $\nu(x)$ from \eqref{PG11}, we get
	\begin{equation}\label{PG12}
		\nu(x)=\frac{a+b}{a-b}\tau'(x)-\frac{2}{a-b}\psi'(x).
	\end{equation}
So we obtain the second relation between the functions $\tau(x)$, $\nu(x)$ in the form \eqref{PG12}. 	

\item[Step 3.]
In order to prove the uniqueness of solution to the Problem \ref{PG02} let consider the homogeneous conditions $\varphi_0(y)\equiv 0$,
$\varphi_1(y)\equiv0 $, $\psi(x)\equiv0$. 
Then $\tau(0)=0$, $\tau(1)=0$.
Let us show that it is possible only for $\tau(x)\equiv0 $.
For $\psi_1(x)\equiv 0$, we can write the equality \eqref{PG12} as
\begin{equation}\label{PG122}
	\nu(x)=\frac{a+b}{a-b}\tau'(x).
\end{equation}
Then
$$
\int\limits_{0}^{1}\tau(\xi)\nu(\xi)d\xi=\frac{a+b}{a-b}\int\limits_{0}^{1}\tau(\xi)\tau'(\xi)d\xi=
$$	
\begin{equation}\label{PG13}	
	=\frac{1}{2}\frac{a+b}{a-b}\int\limits_{0}^{1}[\tau^2(\xi)]'d\xi=\frac{1}{2}\frac{a+b}{a-b}(\tau^2(1)-\tau^2(0))=0.
\end{equation}	
Therefore, $\tau(x)=const$, and since $\tau(0)=0$, then $\tau(x)\equiv0 $.

By \eqref{PG122} then $\nu(x) \equiv0$. From this and \eqref{PG10} we conclude that $u(x,y)\equiv 0$
in $D_2$ for $\varphi_0(y)\equiv 0$,
$\varphi_1(y)\equiv0 $, $\psi_1(x)\equiv0$. Uniqueness in $D_2$ is proved.

\item[Step 4.] In the domain $D_1$, the solution to the problem
$$
u_{xx}- u_y=0,
$$
$$
u(0,y)=\varphi_0(y),\qquad 	u(1,y)=\varphi_1(y),
$$
$$
\lim\limits_{y\rightarrow +0}u(x,y)=\tau(x),\qquad x\in[0,1],
$$
is
(see \cite{Polyanin}, p. 62):
\begin{equation}\label{BZHU10}
	u(x,y){=}\int\limits_0^y  G_\xi(x,y;0,\tau)\varphi_0(\tau)d\tau-\int\limits_0^y  G_\xi(x,y;1,\tau)\varphi_1(\tau)d\tau
	+\int\limits_0^1G(x,y;\xi,0)\tau(\xi) d\xi,
\end{equation}
where
\begin{equation}\label{G0}
	G(x,y;\xi,\eta)=2\sum\limits_{n=1}^\infty \sin(\pi n x)\sin(\pi n \xi)e^{-n^2\pi^2 (y-\eta)}.
\end{equation}	 
This representation of the solution follows that for $\varphi_0(y)\equiv0$, $\varphi_1(y)\equiv0$, $\tau(x)\equiv 0$ in $D_1$ the solution $u(x,y )\equiv0$.
It means that the homogeneous problem has only a trivial solution, therefore, if the solution to the problem \eqref{PG02} exists, then it is unique.
 
	\item[Step 5.] Let us prove the existence of a solution. Substitution $\nu(x)=\tau''(x)$ into \eqref{PG11} gives
 $$
 (a+b)\tau'(x)-(a-b)\tau''(x)=2\psi_1'(x).
 $$
or
 \begin{equation}\label{DP01}
 	\tau''(x)-\frac{a+b}{a-b}\tau'(x)=-\frac{2}{a-b}\psi'(x)
 \end{equation}
We add  conditions
 \begin{equation}\label{DP02}
 	\tau(0)=\varphi_0(0),\qquad \tau(1)=\varphi_1(0).
 \end{equation}
to \eqref{DP01}.
Multiplying \eqref{DP01} on $e^{-\frac{a+b}{a-b}x}$ we obtain
$$
\frac{d}{dx}\left(e^{-\frac{a+b}{a-b}x}\frac{d\tau}{dx} \right) =-\frac{2}{a-b}e^{-\frac{a+b}{a-b}x}\psi'(x).
$$
Since $p(x)=e^{-\frac{a+b}{a-b}x}\in C^1[0,1]$, $f(x)=-\frac{2}{a-b} e^{-\frac{a+b}{a-b}x}\psi'(x)\in C[0,1]\cap C^1(0,1)$ because $\psi (x)\in C^1[0,1]\cap C^2(0,1)$, then $\tau(x)$  exists and $\tau(x)\in C^2[0, 1]$.
Therefore, by \eqref{PG12} $\nu(x) \in C[0,1]\cap C^1(0,1)$ and $u(x,y)\in C(\overline{D} ){\cap} C^2(D)$.
\end{description}
\end{proof}

\begin{example} Solve the equation \eqref{PG01} with conditions
	\begin{equation}\label{Ex01}
		u(0,y)=1-y,\qquad 	u(1,y)=y,\qquad y>0,
	\end{equation}
	\begin{equation}\label{Ex02}
		2u\left(\frac{x}{2},-\frac{x}{2}\right) -u\left(\frac{x+1}{2},\frac{x-1}{2}\right)=4x,\qquad 	0\leq x\leq 1,
	\end{equation}
\end{example}
We have $a=2$, $b=-1$, $\varphi_0(y)=1-y$, $\varphi_1(y)=y$, $\psi(x)=x$.
Equality \eqref{PG05} is valid.
Let's find a solution to the problem
$$
\tau''(x)-\frac{1}{3}\tau'(x)=-\frac{2}{3}, 
$$
$$
\tau(0)=1,\qquad \tau(1)=0.
$$
We obtain
$$
\tau(x)=\frac{1}{\sqrt[3]{e}-1}(2 \sqrt[3]{e} x-2 x-3 e^{x/3}+\sqrt[3]{e}+2)
$$
and
$$
\nu(x)=\frac{1}{3 \left(1-\sqrt[3]{e}\right)}e^{x/3}.
$$
For $y<0$
$$
u(x,y)=\frac{-e^{\frac{x-y}{3}}-2 e^{\frac{x+y}{3}}-2 x+\sqrt[3]{e} (2 x+1)+2}{\sqrt[3]{e}-1},
$$
and for $y<0$ 
$$
u(x,y){=}\int\limits_0^y  G_\xi(x,y;0,\tau)(1-\tau)d\tau-\int\limits_0^y  G_\xi(x,y;1,\tau) \tau d\tau
+\int\limits_0^1G(x,y;\xi,0) \xi d\xi,
$$
where
$$
G(x,y;\xi,\eta)=2\sum\limits_{n=1}^\infty \sin(\pi n x)\sin(\pi n \xi)e^{-n^2\pi^2 (y-\eta)}.
$$	
Therefore
$$
u(x,y){=}2\sum\limits_{n=1}^\infty \pi n\sin(\pi n x)\int\limits_0^y  e^{-n^2\pi^2 (y-\tau)}(1-\tau)d\tau-
2\sum\limits_{n=1}^\infty (-1)^n\pi n\sin(\pi n x)\int\limits_0^y e^{-n^2\pi^2 (y-\tau)} \tau d\tau+
$$
$$
+2\sum\limits_{n=1}^\infty \sin(\pi n x)e^{-n^2\pi^2 y}\int\limits_0^1\sin(\pi n \xi) \xi d\xi=
$$
$$
2\sum\limits_{n=1}^\infty \frac{\sin(\pi n x)}{\pi^3n^3}
\left( 1-\pi ^2 n^2 (y-1)-e^{-\pi ^2 n^2y} (\pi ^2 n^2+1)\right) -
2\sum\limits_{n=1}^\infty (-1)^n\frac{\sin(\pi n x)}{\pi^3 n^3}\left(\pi ^2 n^2 y+e^{-\pi ^2n^2y}\right) -
$$
$$
-2\sum\limits_{n=1}^\infty\frac{(-1)^n}{\pi n} \sin(\pi n x)e^{-n^2\pi^2 y}.
$$
The plot where the first 100 terms of the series are taken is shown in Fig. \ref{Ris1}
We can see that the solution from the domain $D_1$ continuously passes into the solution in the domain $D_2$.
\begin{figure}[h!]
	\center{\includegraphics[width=1\linewidth]{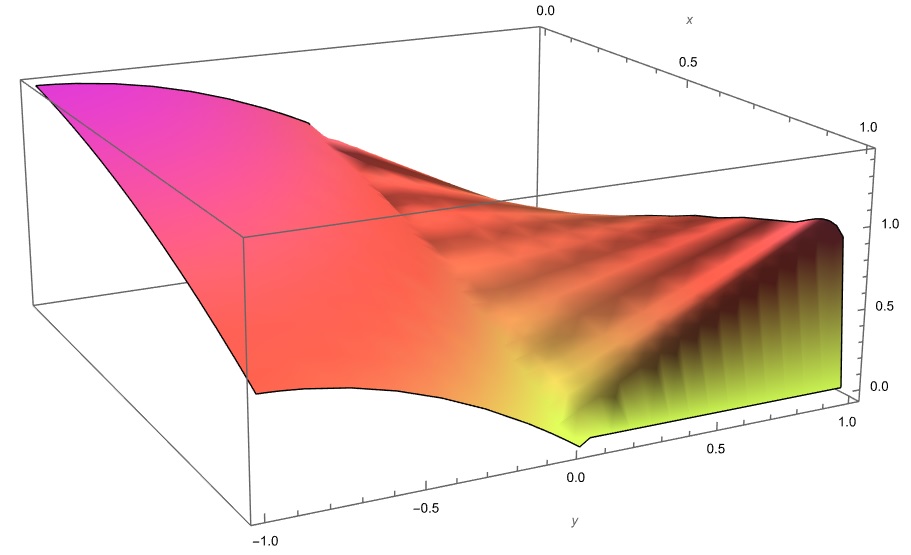}}
	\caption{Solution $u(x,y)$.}\label{Ris1}
\end{figure}

\end{document}